\documentclass[a4paper,10pt]{amsart}

\usepackage[english]{babel}
\usepackage[utf8]{inputenc}
\usepackage[T1]{fontenc}
\usepackage{lmodern}

\usepackage{amssymb}
\usepackage{amsmath}
\usepackage{amsthm}
\usepackage{bbm}

\usepackage[shortlabels]{enumitem}
\usepackage{tikz}
\usepackage{marginnote}
\usepackage{mathtools}
\usepackage{orcidlink}

\usepackage{hyperref}

\newcommand{\bbB}{\mathbb{B}}
\newcommand{\bbC}{\mathbb{C}}

\newcommand{\bbN}{\mathbb{N}}

\newcommand{\bbR}{\mathbb{R}}

\newcommand{\calJ}{\mathcal{J}}

\newcommand{\calL}{\mathcal{L}}

\newcommand{\N}{\bbN}
\newcommand{\R}{\bbR}
\newcommand{\C}{\bbC}

\DeclarePairedDelimiter{\norm}{\lVert}{\rVert}
\DeclarePairedDelimiter{\abs}{\lvert}{\rvert}
\DeclarePairedDelimiter{\dual}{\langle}{\rangle}
\DeclarePairedDelimiter{\set}{\lbrace}{\rbrace}

\DeclareMathOperator{\id}{id}
\DeclareMathOperator{\one}{\mathbbm{1}}
\DeclareMathOperator{\boldOne}{\boldsymbol{1}}

\DeclareMathOperator{\fix}{fix}
\DeclareMathOperator{\nspan}{span}

\DeclareMathOperator{\diag}{diag}
\DeclareMathOperator{\dN}{\calJ}

\newcommand{\ud}{\mathrm{d}}

\newcommand{\ue}{\mathrm{e}}

\DeclareMathOperator{\NB}{NB}
\DeclareMathOperator{\B}{B}
\newcommand{\PhimNB}{(\Phi^-)_{\NB}}
\newcommand{\PhipNB}{(\Phi^+)_{\NB}}
\newcommand{\PhipmNB}{(\Phi^\pm)_{\NB}}
\newcommand{\PhimB}{(\Phi^-)_{\B}}
\newcommand{\PhipB}{(\Phi^+)_{\B}}
\newcommand{\Phiwm}{\Phi^-_{w}}

\newcommand{\PhiwmB}{(\Phiwm)_{\B}}

\newcommand{\BNB}{\bbB_{\NB}}
\newcommand{\BB}{\bbB_{\B}}

\theoremstyle{definition}
\newtheorem{definition}{Definition}[section]
\newtheorem{remark}[definition]{Remark}

\theoremstyle{plain}
\newtheorem{proposition}[definition]{Proposition}
\newtheorem{lemma}[definition]{Lemma}
\newtheorem{theorem}[definition]{Theorem}
\newtheorem{corollary}[definition]{Corollary}

\numberwithin{equation}{section}

\begin{document}

\title[Buffered network flows]{Well-posedness and long-term behaviour of buffered flows in infinite networks}

\author{Alexander Dobrick \orcidlink{0000-0002-3308-3581}}
\address[A.~Dobrick]{Alexander Dobrick, Christian-Albrechts-UniversitÃ¤t~zu~Kiel, Arbeitsbereich~Analysis, 24118 Kiel, Germany}
\email{dobrick@math.uni-kiel.de}

\author{Florian G.\ Martin}
\address[F.~Martin]{Florian G.\ Martin, Institut fÃ¼r Angewandte Analysis, UniversitÃ¤t Ulm, 89069 Ulm, Germany}
\email{florian.martin@alumni.uni-ulm.de}

\date{\today}
\subjclass[2020]{47D06, 35R02, 47D07, 60K20}

\begin{abstract}
	We consider a transport problem on an infinite metric graph and discuss its well-posedness and long-term behaviour under the condition that the mass flow is buffered in at least one of the vertices. In order to show the well-posedness of the problem, we employ the theory of $C_0$-semigroups and prove a Desch--Schappacher type perturbation theorem for dispersive semigroups. Investigating the long-term behaviour of the system, we prove irreducibility of the semigroup under the assumption that the underlying graph is strongly connected and an additional spectral condition on its adjacency matrix. Moreover, we employ recent results about the convergence of stochastic semigroups that dominate a kernel operator to prove that the solutions converge strongly to equilibrium. Finally, we prove that the solutions converge uniformly under more restrictive assumptions. 
\end{abstract}

\maketitle

\section{Introduction} \label{section:introduction}

Flows in infinite networks are an intriguing and important topic of study due to their relevance in a wide range of real-world systems and their inherent complexity. Infinite networks provide a rich framework for modelling and understanding the dynamics of various real-world phenomena as, e.g., fluid flows in porous media. Studying network flows on metric graphs by methods from the theory of strongly continuous semigroups has a history. The seminal results in this direction are due to Kramar and Sikolya \cite{Kramar2005}. They considered a mass distribution on the edges of a finite metric graph which is subject to a transportation process on each edge (described mathematically by a first order partial differential equation) and which is re-distributed to different edges according to pre-defined weights whenever it reaches a vertex of the graph. The long-term behaviour of such a network flow depends crucially on the ratio of the transport velocities on different edges, see \cite[Theorem~4.5]{Kramar2005}. This seminal paper was followed by several further papers on the topic involving absorption processes \cite{Matrai2007}, dynamic ramification nodes \cite{Sikolya2005} and Boltzmann type scattering phenomena \cite{Dorn2010}. In \cite{Dorn2008}, Dorn first considered flows in infinite networks and proved the well-posedness as well as that the solutions behave uniformly asymptotically periodic under suitable assumptions on the underlying graph. Furthermore, the results from \cite{Kramar2005} were extended to infinite networks in \cite{Dorn2009}. Moreover, \cite{Budde2021} and \cite{Dobrick2022} discuss the well-posedness and asymptotic behaviour of network flows in infinite networks on a different phase space within the framework of bi-continuous semigroups. Finally, there are several more papers that study various transport processes on metric graphs (see \cite{Bayazit2013, Engel2008, KunszentiKovacs2009, KunszentiKovacs2009a, Radl2008}). 

In this paper, we consider a network flow in infinite networks in the spirit of \cite{Dorn2008, Dorn2009}. However, we introduce mass buffers in some vertices of the network. Mass buffers absorb and store all mass that reaches the vertex on an incoming edge, and release the stored mass to the outgoing edges at a constant rate proportional to the mass contained in the buffer. In particular, a buffer with a lot of mass in storage releases mass at a higher rate than a buffer which is almost empty. Analogous to unbuffered network flows, buffered network flows can be formulated as an abstract Cauchy problem on an appropriate $L^1$-space (or, more generally, AL-space). However, as it turns out, buffers in the network change the long-term behaviour of the solutions to the transport equation drastically in comparison to the results from \cite{Dorn2008, Dorn2009, Kramar2005}. In particular, under natural assumptions on the underlying metric graph of the network buffers yield the convergence of the solutions to equilibrium and not their mere asymptotic periodicity. 

In the case where the edges of the network have unit length and the velocities are constantly equal to $1$, an analysis of the buffered network flow described in this paper is contained in the second-named authors' Master's thesis \cite{Martin2018}.

\subsection*{Contributions of this article}

In this paper, we prove that the Cauchy problem that describes the buffered network flow is well-posed in the sense of the theory of $C_0$-semigroups (see Theorem~\ref{theorem:well-posedness-of-network-flow-problem}). For that purpose, we first prove several versions of a Desch--Schappacher type perturbation result for non-analytic semigroups (see Section~\ref{section:relatively-compact-perturbations-of-semigroup-generators}). This result is then leveraged to prove the well-posedness of buffered network flows in infinite networks. Finally, we show that, if the underlying graph is strongly connected, then the flow converges strongly or uniformly as time tends to infinity, depending on further assumptions on the graph (see Theorem~\ref{theorem:strong-convergence-of-the-flow} and Theorem~\ref{theorem:norm-convergence-of-the-flow}). These results are consequences of two recent abstract convergence theorems $C_0$-semigroups (see \cite[Theorem~6.1]{Gerlach2019} and \cite[Theorem~1.1]{Glueck2022}). 

\subsection*{Organization of the article}

In Section~\ref{section:a-model-for-a-buffered-transport-problem-on-a-graph}, we recall some notions from graph theory and describe the mathematical model for buffered network flow. Several general perturbation results of semigroup generators are proven in Section~\ref{section:relatively-compact-perturbations-of-semigroup-generators} and applied to our particular model in Section~\ref{section:well-posedness-of-the-transport-problem}. Finally, in Section~\ref{section:long-term-behaviour-of-the-flow}, the long-term behaviour of the flow is discussed.

\section{Buffered transport problem on a graph} \label{section:a-model-for-a-buffered-transport-problem-on-a-graph}

\subsection*{A real-world example} 

The following description of the buffered network flow consists of many definitions, assumptions and conditions, which will become much more intuitive if one keeps in mind a real-world example such as a connected system of pipes with a running fluid in it. Two or more pipes may encounter each other in nodes that distribute all the incoming mass into the outgoing pipes without loss or generation of new mass. We avoid peculiar scenarios like isolated nodes or fluid moving backward in pipes. Fluid velocity may fluctuate due to variations in pipe diameter but remains constant over time, pressure, or flow rate. Some nodes feature buffers, akin to tanks, temporarily storing fluid and releasing it proportionally. Envisioning this, one naturally expects flow equilibrium over time, barring \emph{dead ends} or \emph{one-way routes} hindering balance, as shown in Section~\ref{section:long-term-behaviour-of-the-flow}.

\subsection*{Basic notions from graph theory}

We start this section, by introducing some notation and general assumptions that will be used throughout this article. A \emph{directed graph} $G$ is a pair $(V, E)$ where
\begin{enumerate}[\upshape (i)]
    \item $V$ is a \emph{set of vertices},
    \item $E \subseteq V \times V$ is a \emph{set of directed edges}.
\end{enumerate}
Throughout this paper, we assume that $G = (V, E)$ is a directed graph and that $V$ and $E$ are both non-empty and at most countable. Every edge $e = (v_1,v_2) \in E$ connects two different vertices $v_1 \neq v_2$ and is said to have its \emph{tail} in $v_1 \in V$ and its \emph{head} in $v_2 \in V$. In this situation, the vertex $v_1$ is said to have an \emph{outgoing edge} $e$ and $v_2$ is said to have an \emph{incoming edge} $e$. The length of an edge $e \in E$ is denoted by the real number $l_e > 0$ and we identify $e$ with the interval $[0,l_e]$. By convention, we parameterize the edge $e \in E$ such that its tail is in $l_e$ and its head in $0$. Moreover, we assume the graph $G$ to be:
\begin{enumerate}[\upshape (a)]
    \item \emph{simple}, i.e., $G$ contains no loops and no multiple edges,
    \item \emph{locally finite}, i.e., each vertex $v \in V$ only has finitely many outgoing edges,
    \item \emph{non-degenerate}, i.e., each vertex $v \in V$ has at least one incoming as well as at least one outgoing edge.
\end{enumerate}
Furthermore, every vertex may possess a buffer. Hence, the set of vertices $V = N \sqcup B$ is disjointly partitioned into the set $N$ of vertices \emph{without buffer} and the set $B$ of vertices \emph{with buffer}. Finally, we assume that there exists at least one vertex with buffer, i.e., $B \neq \emptyset$.

\subsection*{Incidence matrices}

The structure of the graph is completely described by \emph{incidence matrices}. To this end, we define the \emph{outgoing incidence matrix} $\Phi^- = (\phi_{ve}^-) \in \R^{V \times E}$ by
\begin{align}
	\label{def:outgoing-incidence-matrix}
	\phi_{ve}^- \coloneqq
	\begin{cases}
		1, & \text{ if the vertex } v \text{ is the tail of the edge } e, \\
		0, & \text{ otherwise},
	\end{cases}
\end{align}
and the \emph{incoming incidence matrix} $\Phi^+ = (\phi_{ve}^+) \in \R^{V \times E}$ by
\begin{align}
	\label{def:incoming-incidence-matrix}
	\phi_{ve}^+ \coloneqq
	\begin{cases}
		1, & \text{ if the vertex } v \text{ is the head of the edge } e, \\
		0, &\text{ otherwise}.
	\end{cases}
\end{align}
Note that both $\Phi^-$ and $\Phi^+$ have exactly one non-zero entry in each column. After rearranging, we can assume that the upper lines of the incidence matrices refer to the vertices \emph{without} buffer and the bottom lines to the vertices \emph{with} buffer. This allows us to subdivide the incidence matrices into
\begin{align*}
	\Phi^- =
	\begin{pmatrix} 
		\PhimNB \\
		\PhimB
	\end{pmatrix}
\end{align*}
and
\begin{align*}
	\Phi^+ = 
	\begin{pmatrix}
		\PhipNB \\
		\PhipB
	\end{pmatrix}
\end{align*}
with the submatrices $\PhimNB \in \R^{N \times E}$ and $\PhipNB \in \R^{N \times E}$ (referring to the vertices which have no buffer) as well as $\PhimB \in \R^{B \times E}$ and $\PhipB \in \R^{B \times E}$ (referring to the vertices which have buffers). Observe that if $B$ would be empty, then $\PhipmNB$ would coincide with $\Phi^{\pm}$, respectively.

\subsection*{Weights}

Furthermore, we assume that the edges of the graph $G$ are \emph{weighted} by weights $w_{v e} \in [0, 1]$, $v \in E$ and $e \in E$, with the property that
\begin{align} \label{equation:normalization-of-weights}
    \sum_{e \in E} w_{v e} = 1.
\end{align}
Clearly, we assume $w_{ve} = 0$ if $e$ is not an outgoing edge of $v$. However, if $e$ is an outgoing edge of $v$, then we assume $0 < w_{ve} \leq 1$. The normalisation condition \eqref{equation:normalization-of-weights} will ensure that transport processes on the network are conservative, i.e., that their solutions are given by stochastic semigroups. To employ these weights in the model of the network, one considers the \emph{weighted outgoing incidence matrix} $\Phi_w^- =(\phi_{w, v e}^-)_{v \in V, e \in E} \in \R^{V \times E}$ given by
\begin{align*}
    \phi_{w, v e}^- \coloneqq 
    \begin{cases}
        w_{ve}, &\quad \text{if } v \text{ is the tail of } e, \\
        0, &\quad \text{else}.
    \end{cases}
\end{align*}
Analogously to the decomposition of the incidence matrices, we subdivide the weighted incidence matrix $\Phi_w^- =(\phi_{w, v e}^-)_{v \in V, e \in E}$ into matrices $(\Phi_w^-)_{\mathrm{NB}} \in \R^{N \times E}$ and $(\Phi_w^-)_{\mathrm{B}} \in \R^{B \times E}$ and set
\begin{align*}
    \BNB \coloneqq ((\Phi_w^-)_{\mathrm{NB}})^T(\Phi^+)_{\mathrm{NB}} \in \R^{E \times E}
\end{align*}
and 
\begin{align*}
    \BB \coloneqq ((\Phi_w^-)_{\mathrm{B}})^T \in \R^{E \times B}.
\end{align*}
An alternative way to describe the structure of a given graph is by the \emph{weighted (transposed) adjacency matrix} $\bbB = (\bbB_{e k})_{e, k \in E}$ defined by $\bbB \coloneqq (\Phi_w^-)^T \Phi^+ \in \R^{E \times E}$. Its entries are given by
\begin{align*}
    \bbB_{e k} = 
    \begin{cases}
        w_{ve}, &\quad  \text{if } v \text{ is the head of } k \text{ and the tail of } e, \\
        0, &\quad \text{else}.
    \end{cases}
\end{align*}
As a direct consequence of \eqref{equation:normalization-of-weights}, $\bbB$ is a column-stochastic matrix and, thus, can be identified with a stochastic operator on the space $\ell^1(E)$. Moreover, due to the decomposition of the weighted incidence matrix, one has $\bbB = \BNB + \BB (\Phi^+)_{\mathrm{B}}$.

\subsection*{The flow equation}

We denote the mass distribution on an edge $e \in E$ at the position $x \in [0,l_e]$ and time $t \geq 0$ by $u_e(x,t)$. Moreover, the time-independent velocity of the flow on an edge $e \in E$ at $x \in [0,l_e]$ is denoted by $c_e(x)$. We assume that every function $c_e$ is in $W^{1,1} \big( (0,l_e) \big)$ and \emph{strictly positive} almost everywhere. In particular, the flow cannot change its direction. Note further that $W^{1,1} ((a,b))$ compactly embeds into the space $C([a,b])$ for all $a, b \in \R$ such that $a < b$. So we assume that the velocities are uniformly bounded in the sense that $\sup_{e \in E} \norm{c_e}_{W^{1,1}} < \infty$ and
\begin{align*}
    0 < c_{\min} \leq  c_e(x) \leq c_{\max} < \infty \qquad \text{for all } e \in E, \, x \in [0, l_e],
\end{align*}
where $0 < c_{\min} \leq c_{\max} < \infty$ are constants.
	
A buffer in a vertex $v \in B$ accumulates the incoming mass and emits an amount proportional to its stored mass. We denote the content at the time $t \geq 0$ by $b_v(t)$ and the positive proportionality constant by $k_v > 0$. Finally, the flow on the network can be described by the following system of equations:
\begin{equation} \label{eq:the-equation}
	\left\{
	\begin{aligned}
		\frac{\partial}{\partial t} u_e(x,t) &= \frac{\partial}{\partial x} \big(c_e(x) u_e(x,t) \big), && e \in E, \, x \in (0,l_e), \\
		\frac{\partial}{\partial t} b_v(t) &= -k_v b_v(t)+ \sum_{e \in E} \phi_{ve}^+ c_e(0) u_e(0,t), && v \in B, \\
        \phi_{ve}^- c_e(l_e) u_e(l_e,t) &= w_{ve} k_v b_v(t), && v \in B, \\
		\phi_{ve}^- c_e(l_e) u_e(l_e,t) &= w_{ve} \sum_{k \in E} \phi_{vk}^+ c_k(0) u_k(0,t), && v \in N, \\
        u_e(x,0) &= g_e(x), && e \in E, \, x \in (0,l_e), \\ 
        b_v(0) &= h_v, && v \in B. 
	\end{aligned}
	\right.
\end{equation}
The differential equations in the first two lines determine the overall movement of the mass: On the edges the flow is modelled by a simple linear transport equation with a position-dependent velocity. In addition, we state a second differential equation for the buffers, stating that the rate of change is equal to the sum of all incoming material
\begin{align*}
	\sum_{e \in E} \phi_{ve}^+ c_e(0) u_e(0,t)
\end{align*}
reduced by the total outgoing material
\begin{align*}
	k_v b_v(t) = \sum_{e \in E} w_{ve} k_v b_v(t) = \sum_{e \in E} \phi_{ve}^- c_e(l_e) u_e(l_e,t).
\end{align*}
portioned into all outgoing edges according to the boundary conditions.

For vertices without a buffer, the sum of the boundary conditions for all edges shows the equality of incoming and outgoing mass according to the generalized \emph{Kirchhoff law}
\begin{align*}
	\sum_{e \in E} \phi_{ve}^- c_e(l_e) u_e(l_e,t) = \sum_{e \in E} \phi_{ve}^+ c_e(0) u_e(0,t).
\end{align*}

\section{Relatively compact perturbations of semigroup generators} \label{section:relatively-compact-perturbations-of-semigroup-generators}

In this section, we prove several Desch--Schappacher type perturbation theorems. Let $(T(t))_{t \geq 0}$ be a $C_0$-semigroup on a Banach space $X$ with generator $A$. In what follows, we endow $D(A)$ with its graph norm, denoted by $\norm{\, \cdot \,}_{D(A)}$. If $(T(t))_{t \geq 0}$ is analytic and $B \colon D(A) \to X$ is compact, then it was proven by Desch and Schappacher in \cite[Theorem~1]{Desch1988} that $A + B$ generates an analytic semigroup, too. First, we replace the analyticity of $(T(t))_{t \geq 0}$ with the assumption that $A + B$ is dissipative. Recall that $A$ is \emph{dissipative} if for every $f \in D(A)$ one has
\begin{align*}
    \dual{Af, \psi} \leq 0 \qquad \text{for one/all } \ \psi\in \dN(f),
\end{align*}
where $\dN(f) \coloneqq \set{\psi \in X' : \norm \psi \leq 1 \text{ and } \dual{f, \psi} = \norm f}$ denotes the \emph{duality set} of $f$. Dissipativity plays a major role in Lumer--Phillips theorem \cite[Theorem~II.3.15]{Engel2000} characterizing contractive $C_0$-semigroups. Our argument is very similar to the one used by Desch and Schappacher. Its essence is already contained in the following proposition.

\begin{proposition} \label{prop:resolvent-set-perturbation} 
	Let $A$ be a densely defined linear operator on a Banach space $X$ and let $B \colon D(A) \to X$ be a compact operator. Suppose that there exists $\lambda_0 > 0$ such that $[\lambda_0,\infty) \subseteq \rho(A)$ and
	\begin{align*}
		\sup_{\lambda \in [\lambda_0,\infty)} \norm{\lambda R(\lambda, A)} < \infty.
	\end{align*}
	If $\lambda > \lambda_0$ is sufficiently large, then $\lambda \in \rho(A + B)$ and
	\begin{align} \label{eq:perturbed-resolvent}
		R(\lambda, A + B) = \bigg(\sum_{k=0}^\infty (R(\lambda, A) B)^k \bigg) R(\lambda, A), 
	\end{align}
	where the series is convergent in $\calL(D(A))$.
\end{proposition}

We shall make use of the following elementary lemma in the proof of Proposition~\ref{prop:resolvent-set-perturbation}.

\begin{lemma} \label{lemma:resolvent-estimates}
    Let $A$ be a densely defined linear operator on a Banach space $X$ such that there exists $\lambda_0 > 0$ such that $[\lambda_0,\infty) \subseteq \rho(A)$ and 
    \begin{align*}
		M \coloneqq \sup_{\lambda \in [\lambda_0,\infty)} \norm{\lambda R(\lambda, A)} < \infty.
	\end{align*}
    Then the following assertions hold:
    \begin{enumerate}[\upshape (i)]
        \item $\limsup_{\lambda \to \infty} \norm{R(\lambda, A)}_{X \to D(A)} < \infty$. 
        \item $\lim_{\lambda \to \infty} \norm{R(\lambda, A)}_{D(A) \to D(A)} = 0$. 
    \end{enumerate}
\end{lemma}
\begin{proof}
    (i): Let $x \in X$ such that $\norm{x} = 1$ and $\lambda > \lambda_0$. First notice that 
    \begin{align}
        \norm{R(\lambda, A) x} = \frac{1}{\lambda} \norm{\lambda R(\lambda, A) x} \leq \frac{M}{\lambda} \norm{x}. \label{eq:resolvent-estimate-simple}
    \end{align}
    Moreover, one has
    \begin{align*}
        \norm{A R(\lambda, A) x} = \norm{x - \lambda R(\lambda, A) x} \leq (1 + M) \norm{x}.
    \end{align*}
    Consequently, we obtain
    \begin{align*}
        \norm{R(\lambda, A) x}_{D(A)} = \norm{A R(\lambda, A) x} + \norm{R(\lambda, A) x} \leq \bigg((1 + M) + \frac{M}{\lambda}\bigg) \norm{x},
    \end{align*}
    which yields the claim. 

    (ii): Let $x \in D(A)$ with $\norm{x}_{D(A)} = 1$ and $\lambda > \lambda_0$. Then one has 
    \begin{align*}
        \norm{A R(\lambda, A) x} = \norm{R(\lambda, A) A x} \leq \norm{R(\lambda, A)} \cdot \norm{Ax} \leq \frac{M}{\lambda} \cdot \norm{A x}.
    \end{align*}
    This, together with~\eqref{eq:resolvent-estimate-simple} yields
    \begin{align*}
        \norm{R(\lambda, A)x}_{D(A)} &= \norm{A R(\lambda, A) x} + \norm{R(\lambda, A) x} \\
        &\leq \frac{M}{\lambda} (\norm{Ax} + \norm{x}) = \frac{M}{\lambda} \norm{x}_{D(A)},
    \end{align*}
    which yields the claim. 
\end{proof}

\begin{proof}[Proof~of~Proposition~\ref{prop:resolvent-set-perturbation}]
    Clearly, one has
    \begin{align} \label{eq:neumann-series}
        \lambda - (A + B) = (\lambda - A)(\id_{D(A)} - R(\lambda, A)B) \qquad \text{for all } \lambda \in \rho(A).
    \end{align}
    So it suffices to prove that the norm of the operator $R(\lambda, A) B \colon D(A) \to D(A)$ is strictly smaller than $1$ for all sufficiently large real numbers $\lambda$. By Lemma~\ref{lemma:resolvent-estimates}, we have 
    \begin{align*}
        \limsup_{\lambda \to \infty} \norm{R(\lambda, A)}_{X \to D(A)} < \infty \quad \text{and} \quad \lim_{\lambda \to \infty} \norm{R(\lambda, A)}_{D(A) \to D(A)} = 0. 
    \end{align*}
    Since $D(A)$ is dense in $X$, this implies that $R(\lambda, A)$ converges to $0$ with respect to the strong operator topology in $\calL(X; D(A))$ as $\lambda \to \infty$.
	
	Denote the closed unit ball in $D(A)$ by $U$. Due to the compactness of $B$ the set $B(U)$ is relatively compact in $X$ and, hence,
    \begin{align*}
     \norm{R(\lambda, A)B}_{D(A) \to D(A)} = \sup_{x \in U} \norm{R(\lambda, A) B x}_{D(A)} \to 0 \qquad \text{as } \lambda \to \infty.
    \end{align*}
	Hence, $\norm{R(\lambda, A) B} < 1$ for all sufficiently large $\lambda > \lambda_0$ and the identity~\eqref{eq:perturbed-resolvent} follows from~\eqref{eq:neumann-series} and a standard Neumann series argument.
\end{proof}

As a consequence of the proposition above, we obtain the following perturbation result.

\begin{theorem} \label{theorem:compact-perturbation-dissipative}
	Let $(T(t))_{t \geq 0}$ be a $C_0$-semigroup on a Banach space $X$ with generator $A$ and let $B \colon D(A) \to X$ be a compact operator. If $A + B$ is dissipative, then $A + B$ generates a contractive $C_0$-semigroup on $X$.
\end{theorem}

\begin{proof}
	By rescaling, we can assume without loss of generality that $(T(t))_{t \geq 0}$ is bounded. Then this follows immediately from Proposition~\ref{prop:resolvent-set-perturbation} and the Lumer--Phillips theorem \cite[Theorem~II.3.15]{Engel2000}.
\end{proof}

\begin{remark}
	Theorem~\ref{theorem:compact-perturbation-dissipative} is quite similar to \cite[Corollary~III.2.17(i)]{Engel2000}. However, there are two main differences. Namely, in \cite[Corollary~III.2.17(i)]{Engel2000} it is assumed that:
	\begin{enumerate}[\upshape (a)]
		\item The space $X$ is reflexive or that the operator $B$ is closable in $X$. Theorem~\ref{theorem:compact-perturbation-dissipative} shows that those assumptions are not necessary.
		\item Both operators $A$ and $B$ are dissipative. This is more restrictive than the assumption that merely $A + B$ is dissipative. Indeed, the buffered network flow discussed in Section~\ref{section:well-posedness-of-the-transport-problem} constitutes an example, where both the operators $A$ and $A + B$ are dissipative, but the operator $B$ is not.
	\end{enumerate}
\end{remark}

Recall that $A$ on a Banach lattice $X$ is called \emph{dispersive} if for every $f \in D(A)$ one has
\begin{align*}
    \dual{Af, \psi} \leq 0 \qquad \text{for one/all } \ \psi\in \dN^+(f),
\end{align*}
where $\dN^+(f) \coloneqq \set{\psi \in X_+' : \norm \psi \leq 1 \text{ and } \dual{f, \psi} = \norm{f^+}}$ denotes the \emph{positive duality set} of $f$. By replacing the dissipativity in Theorem~\ref{theorem:compact-perturbation-dissipative} with dispersivity one obtains the following perturbation result.  

\begin{theorem} \label{theorem:compact-perturbation-dispersive}
	Let $(T(t))_{t \geq 0}$ be a $C_0$-semigroup on a Banach lattice $X$ with generator $A$ and let $B \colon D(A) \to X$ be a compact operator. If $A + B$ is dispersive, then $A + B$ generates a positive, contractive $C_0$-semigroup on $X$.
\end{theorem}

\begin{proof}
	By rescaling, we can assume without loss of generality that $(T(t))_{t \geq 0}$ is bounded. This follows immediately from Proposition~\ref{prop:resolvent-set-perturbation} and Phillips' theorem \cite[Theorem~11.10]{Batkai2017}.
\end{proof}

We proceed with a corollary which deals with positive semigroups on AL-spaces and is specifically tailored for our purposes in Section~\ref{section:well-posedness-of-the-transport-problem}. Recall that a Banach lattice $X$ is called an \emph{AL-space} if
\begin{align*}
    \norm{f + g} = \norm f + \norm g \qquad \text{for all } f, g \in X_+.
\end{align*}
Note that each AL-space is isometrically lattice isomorphic to some $L^1$-space (see \cite[Theorem~9.33]{Aliprantis1999} or \cite[Theorem~4.27]{Aliprantis1985}). Furthermore, a Banach lattice $X$ is called an \emph{AM-space} if 
\begin{align*}
    \norm{f \vee g} = \max \set{\norm f, \norm g} \qquad \text{for all } f, g \in X_+.
\end{align*}
Recall that if $X$ is an AL-space, then $X'$ is an AM-space with order unit (cf.~\cite[Theorem~3.3]{Abramovich2002} or \cite[Theorem~4.23]{Aliprantis1985}), which will be denoted by $\one$ (see \cite[Section~9.5]{Aliprantis1999} for more information). Moreover, in this case one has $\norm{f} = \dual{f, \one}$ for all $f \in X_+$ A positive operator $T$ on an AL-space $X$ is called \emph{stochastic} if $\norm{T f} = \norm{f}$ for all $f \in X_+$ or, equivalently, if $T' \one = \one$. A $C_0$-semigroup $(T(t))_{t \geq 0}$ on an AL-space $X$ is called \emph{stochastic} if $T(t)$ is stochastic for each $t \geq 0$. 

\begin{corollary} \label{corollary:compact-perturbation-stochastic}
	Let $(T(t))_{t \geq 0}$ be a $C_0$-semigroup on an AL-space $X$ with generator $A$ and let $B \colon D(A) \to X$ be a compact operator. If $\one \in \ker (A + B)'$, then $A + B$ generates a stochastic $C_0$-semigroup on $X$.
\end{corollary}

\begin{proof}
    By an rescaling argument, we can assume that the semigroup $(T(t))_{t \geq 0}$ is bounded. Let $f \in D(A)$. Then $\one \in \dN^+(f)$ and $\dual{(A + B) f, \one} = \dual{f, (A + B)' \one} = 0$. Thus, $A + B$ is dispersive and Theorem~\ref{theorem:compact-perturbation-dispersive} implies that $A + B$ generates a positive, contractive $C_0$-semigroup $(S(t))_{t \geq 0}$ on $X$. Moreover, \cite[Lemma~II.1.3(iv)]{Engel2000} implies
    \begin{align*}
        \norm{S(t) f} - \norm{f} = \dual{S(t) f - f, \one} = \dual[\bigg]{(A + B) \int_0^t S(s) f \, \ud s, \one} = 0
    \end{align*}
    for all $f \in X_+$. Thus, $(S(t))_{t \geq 0}$ is stochastic.
\end{proof}

We conclude this section with a brief detour to positive semigroups on spaces of continuous functions; the following perturbation result is not needed in the rest of the paper, but it is a nice consequence of Proposition~\ref{prop:resolvent-set-perturbation}, and we find it interesting in its own right. Let $K$ be a compact Hausdorff space and let $C(K)$ denote the space of all scalar-valued continuous functions on $K$, endowed with the supremum norm. An operator $A$ on $C(K)$ is called \emph{resolvent positive} if there exists $\omega \in \R$ such that $(\omega,\infty) \subseteq \rho(A)$ and such that $R(\lambda, A)$ is a positive operator for each $\lambda > \omega$. For a discussion of related results, we refer to \cite[Section~2~and~3]{Arendt1987}, in particular to Corollary~2.4 and Theorem~3.1 therein. It is known that the operator $A$ on $C(K)$ generates a positive $C_0$-semigroup if and only if it is densely defined and resolvent positive, see e.g.\ \cite[Theorem~B-II-1.8]{NagelEd1986}). Hence, Proposition~\ref{prop:resolvent-set-perturbation} implies the following perturbation result.

\begin{theorem} \label{theorem:generator-theorem-ck}
	Let $A$ be the generator of a positive $C_0$-semigroup on $C(K)$ for some compact Hausdorff space $K$. If $B: D(A) \to C(K)$ is compact and positive, then $A+B$ generates a positive $C_0$-semigroup on $C(K)$, too.
\end{theorem}

\begin{proof}
	We use the characterization of generators of positive $C_0$-semigroups on $C(K)$ that we discussed right before the theorem. As $A$ generates a positive $C_0$-semigroup, it is resolvent positive; hence, it follows from Proposition~\ref{prop:resolvent-set-perturbation} that $A+B$ is resolvent positive, too.
\end{proof}

Clearly, $C(K)$ is an AM-space. However, it is unclear to the authors if Theorem~\ref{theorem:generator-theorem-ck} still holds on general AM-spaces, as \cite[Theorem~B-II-1.8]{NagelEd1986} does not hold in this more general context (see \cite[Remark~on~p.~127]{NagelEd1986}). 

\section{Well-posedness of the transport problem} \label{section:well-posedness-of-the-transport-problem}

In this section, we apply Corollary~\ref{corollary:compact-perturbation-stochastic} to the buffered network flow from the Section~\ref{section:a-model-for-a-buffered-transport-problem-on-a-graph}. Observe first that, without loss of generality, one can assume that $l_e = 1$ for all $e \in E$ throughout this section, since \eqref{eq:the-equation} is actually equivalent to the system
\begin{equation} \label{eq:the-equation-normalized}
	\left\{
	\begin{aligned}
		\frac{\partial}{\partial t} u_e(x,t) &= \frac{\partial}{\partial x} (\widetilde{c_e}(x) u_e(x,t)),
		&& e \in E, \, x \in (0, 1), \\
		\frac{\partial}{\partial t} b_v(t) &= -k_v b_v(t)+ \sum_{e \in E} \phi_{ve}^+ \widetilde{c_e}(0) u_e(0, t), \quad && v \in B, \\
        \phi_{ve}^- \widetilde{c_e}(1) u_e(1, t) &= w_{ve} k_v b_v(t), && v \in B, \\
		\phi_{ve}^- \widetilde{c_e}(1) u_e(1, t) &= w_{ve} \sum_{k \in E} \phi_{vk}^+ \widetilde{c_k}(0) u_k(0,t), && v \in N, \\
        u_e(x,0) &= g_e(x), && e \in E, \, x \in (0, 1), \\
        b_v(0) &= h_v, && v \in B,
	\end{aligned}
	\right.
\end{equation}
where $\widetilde{c_e} \coloneqq l_e c_e$ for all $e \in E$. So, in what follows, we choose the state space $X \coloneqq \ell^1(E; L^1([0, 1])) \times \ell^1(B)$ and endow it with its canonical AL-norm
\begin{align*}
    \norm{f} \coloneqq \sum_{e \in E} \norm{u_e}_1 + \sum_{v \in B} \abs{b_v} \qquad \text{for each } f = ((u_e)_{e \in E}, (b_v)_{v \in B}) \in X.
\end{align*}
Recall the boundary conditions of the flow equations \eqref{eq:the-equation-normalized}, given by
\begin{equation}
	\label{equation:boundary-conditions-buffer}
	\phi_{ve}^- c_e(1) u_e(1) = w_{ve} k_v b_v \qquad \text{for all } v \in B
\end{equation}
for vertices \emph{with buffer} and 
\begin{equation}
	\label{equation:boundary-conditions-no-buffer}
	\phi_{ve}^- c_e(1) u_e(1) = w_{ve} \sum_{k \in E} \phi_{vk}^+ c_k(0) u_k(0) \qquad \text{for all } v \in N
\end{equation}
for vertices \emph{without buffer}, respectively. We consider the closed operator
\begin{align*}
    Af \coloneqq \bigg(\bigg(\frac{\ud}{\ud x}(c_e u_e) \bigg)_{e \in E}, \bigg(-k_v b_v + \sum_{e \in E} \phi_{ve}^+ c_e(0) u_e(0) \bigg)_{v \in B} \bigg)
\end{align*}
on $X$ with domain
\begin{equation*}
    \begin{aligned}
        D(A) \coloneqq \set{f = ((u_e)_{e \in E}, \, (b_v)_{v \in B}) \in X &: \\ & \hspace{-2cm}u_e \in W^{1,1} ((0, 1)) \text{ satisfies } \eqref{equation:boundary-conditions-buffer} \text{ and } \eqref{equation:boundary-conditions-no-buffer} \text{ for all } e \in E}.
    \end{aligned}
\end{equation*}

First, we provide an equivalent characterization of the boundary conditions in matrix form.

\begin{lemma}
	\label{lemma:characterization-domain-of-A}
	Let $f = ((u_e)_{e \in E}, \, (b_v)_{v \in B}) \in X$ such that $u_e \in W^{1,1}((0, 1))$ for all $e \in E$. Then $f \in D(A)$ if and only if the identity
	\begin{align}
		\label{equation:boundary-conditions-matrix-form}
		(c_e(1) u_e(1))_{e \in E} = \BNB (c_e(0) u_e(0))_{e \in E} + \BB (k_v b_v)_{v \in B}
	\end{align}
    holds. 
\end{lemma}
\begin{proof}
    Let $f = ((u_e)_{e \in E}, \, (b_v)_{v \in B}) \in X$ such that $u_e \in W^{1,1}((0, 1))$ for all $e \in E$. We show that $f$ satisfies the boundary conditions \eqref{equation:boundary-conditions-buffer} and \eqref{equation:boundary-conditions-no-buffer} if and only the identity  \eqref{equation:boundary-conditions-matrix-form} holds. To this end, we compute the term on the right side of \eqref{equation:boundary-conditions-matrix-form}. By definition, one has
    \begin{align*}
        \BNB (c_e(0) u_e(0))_{e \in E} + \BB (k_v b_v)_{v \in B} = (\Phiwm)^T \begin{pmatrix} \PhipNB (c_e(0) u_e(0))_{e \in E} \\ (k_v b_v)_{v \in B} \end{pmatrix}.
    \end{align*}
    Let $e \in E$ and let $v \in V$ be the tail of $e$. Then $w_{ve} > 0$ is the only non-zero element in the $e$th row of the matrix $(\Phiwm)^T$. Thus, the above identity yields
    \begin{align*}
        (\BNB (c_e(0) u_e(0))_{e \in E} + \BB (k_v b_v)_{v \in B})_e
    	= \begin{cases}
    		w_{ve} k_v b_v, & \text{ if } v \in B, \\
            w_{ve} \sum_{k \in E} \phi_{vk}^+ c_k(0) u_k(0), & \text{ if  } v \in N.
    	\end{cases}
    \end{align*}
    On the other hand, $c_e(1) u_e(1) = \phi_{v e}^- c_e(1) u_e(1)$. Thus, the identity~\eqref{equation:boundary-conditions-matrix-form} is equivalent to the boundary conditions \eqref{equation:boundary-conditions-buffer} and \eqref{equation:boundary-conditions-no-buffer}.
\end{proof}

\begin{remark}
	\label{remark:characterization-domain-of-A-without-buffers}
	If there would not be any vertices with buffer (or, equivalently, if the boundary conditions \eqref{equation:boundary-conditions-no-buffer} would hold true for \emph{all} vertices $v \in V$), then, with the same notation and assumptions as in the above, one would have $f \in D(A)$ if and only if
	\begin{align*}
	    (c_e(1) u_e(1))_{e \in E} = \bbB (c_e(0) u_e(0))_{e \in E}.
	\end{align*}
    This can be seen as in the proof of Lemma~\ref{lemma:characterization-domain-of-A} (cf.~also~\cite[Proposition~3.1]{Dorn2008}).
\end{remark}

We progress by splitting the operator $A$ up into the operator $C$ given by
\begin{align*}
	Cf \coloneqq \bigg(\bigg(\frac{\ud}{\ud x}(c_e u_e)\bigg)_{e \in E}, (-k_v b_v)_{v \in B} \bigg), \qquad D(C) \coloneqq D(A)
\end{align*}
and the perturbation $A - C$ given by
\begin{equation}
	\label{equation:definition-perturbation}
	(A - C)f = \bigg(0, \bigg(\sum_{e \in E} \phi_{ve}^+ c_e(0) u_e(0) \bigg)_{v \in B} \bigg), \qquad D(A - C) = D(A).
\end{equation}

First, we prove that $C$ generates a positive, contractive $C_0$-semigroup due to Phillips' theorem (see e.g.\ \cite[Theorem 11.10]{Batkai2017} or \cite[Theorem~C-II.1.2]{NagelEd1986}). 

\begin{lemma} \label{lemma:C-is-dispersive}
	The operator $C$ is dispersive on $X$.
\end{lemma}

\begin{proof}
	Let $f = \big( (u_e)_{e \in E}, (b_v)_{v \in B} \big) \in D(A)$. Then 
    \begin{align*}
        \psi \coloneqq ((\one_{[u_e > 0]})_{e\in E},(\one_{[b_v > 0]})_{v\in B}) \in \dN^+(f) \subseteq X',
    \end{align*}
    where $X' = \ell^\infty(E; L^\infty([0,1])) \times \ell^\infty(B)$, and one has
	\begin{align*}
		\dual{Cf, \psi} &= \sum_{e \in E} \int_0^1 \frac{\ud}{\ud x} (c_e u_e)(x) \one_{[u_e > 0]}(x) \, \ud x - \sum_{v \in B} k_v b_v \one_{b_v>0} \\
		&=\sum_{e\in E} [c_e(1) u_e(1)]^+ - \sum_{e\in E} [c_e(0) u_e(0)]^+ - \sum_{e\in E}\sum_{v \in B} w_{ve} k_v [b_v]^+.
	\end{align*}
	With Lemma~\ref{lemma:characterization-domain-of-A} we can rephrase the last line in matrix notation and proceed as 
	 \begin{align*}
		\dual{Cf, \psi} &= \dual{[\BNB cu(0) + \BB kb]^+,\boldOne} - \dual{[cu(0)]^+,\boldOne} - \dual{(\PhiwmB )^T [kb ]^+,\boldOne} \\
		&\leq \dual{\bbB [cu(0)]^+,\boldOne} - \dual{[cu(0)]^+,\boldOne},
	\end{align*}
	where $\boldOne \in \ell^\infty(E)$ is the constant one vector. Here, the above inequalities follow immediately from the definitions of the appearing matrices and from the fact that they contain only positive elements. Finally, we conclude that
	\begin{align*}
	\dual{Cf, \psi} &\leq \dual{\bbB [cu(0)]^+,\boldOne} - \dual{[cu(0)]^+,\boldOne} \\
	&= \dual{(\bbB - \id ) [cu(0)]^+, \boldOne} = \dual{[cu(0)]^+, \bbB^T \boldOne - \boldOne} = 0,
	\end{align*}
	since $\bbB$ is column-stochastic. Thus, $C$ is dispersive. 
\end{proof}

Next, we show that $\lambda - C$ is surjective for all sufficiently large $\lambda > 0$. In the proof, for each $y \in \ell^\infty(E)$ we denote its associated diagonal operator by
\begin{align*}
    \diag(y) \colon \ell^1(E) \to \ell^1(E), \quad \diag(y) (r_e)_{e \in E} \coloneqq (y_e r_e)_{e \in E}.
\end{align*}
We will use the following elementary statement on diagonal operators on $\ell^1$. 

\begin{lemma} \label{lemma:bounded-below}
    Let $T$ be a bounded operator on $\ell^1(E)$ and let $f \in \ell^\infty(E)$ be such that $\inf_{e \in E} f_e > \norm{T}$. Then the operator $\diag(f) - T \in \calL(\ell^1(E))$ is invertible. 
\end{lemma}

\begin{proof}
    Clearly, the operator $\diag(f)$ is invertible. So, $\diag(f) - T$ is invertible if and only if $1 - \diag(1/f)T$ is invertible since
    \begin{align*}
        \diag(f) - T = \diag(f) (1 - \diag(1/f)T).
    \end{align*}
    Moreover,
    \begin{align*}
        \norm{\diag(1/f)T} \leq \norm{\diag(1/f)} \cdot \norm{T} \leq \frac{1}{\inf_{e \in E} f_e} \norm{T} < 1.
    \end{align*}
    Thus, the claim follows from a standard Neumann series argument. 
\end{proof}

\begin{lemma} \label{lemma:lambda-C-is-surjective}
	The operator $\lambda - C \colon D(A) \to X$ is surjective for all sufficiently large $\lambda > 0$. In particular, for any $g = (w,z) = ((w_e)_{e \in E}, (z_v)_{v \in B}) \in X$, a solution $f = ((u_e)_{e \in E}, (b_v)_{v \in B}) \in D(A)$ of the equation $(\lambda - C)f = g$ is given by
    \begin{align} \label{eq:resolvent-c}
        f = \bigg(\bigg(\ue^{D_{\lambda, e}(x)} \mu_e + \ue^{D_{\lambda, e}(x)} \int_x^1 \frac{w_e(y)}{c_e(y)} \ue^{-D_{\lambda, e}(y)} \, \ud y \bigg)_{e \in E}, \bigg(\frac{z_v}{\lambda + k_v} \bigg)_{v \in B} \bigg),
    \end{align}
	where $d_{\lambda, e} \coloneqq \frac{\lambda - c_e'}{c_e}$, $D_{\lambda, e}(x) \coloneqq \int_0^x d_{\lambda, e}(y) \, \ud y$ and where $\mu = (\mu_e)_{e \in E} \in \ell^1(E)$ is some constant vector.
\end{lemma}

\begin{proof}
    It is easy to verify that $u_e$ as defined above is contained in $W^{1,1}((0,1))$ and that $u_e$ satisfies the differential equation $\lambda u_e - (c_e u_e)' = w_e$ in the first component of $(\lambda - C)f = g$ for each $e \in E$ (and independent of the value of $\mu$). Furthermore, an easy calculation shows that $b_v$ as defined above satisfies the equation in the second component of $(\lambda - C)f = g$ for every $v \in B$. Hence, it only remains to prove that there exists a vector $\mu \in \ell^1(E)$ such that $f \in D(A)$. 
 
    By Lemma~\ref{lemma:characterization-domain-of-A}, one has $f \in D(A)$ if and only if
    \begin{align*}
        (c_e(1) u_e(1))_{e \in E} = \BNB \cdot (c_e(0) u_e(0))_{e \in E} + \BB \cdot (k_v b_v)_{v \in B}.
    \end{align*}
	Observe that $u_e(1) = \ue^{D_{\lambda, e}(1)} \mu_e$ and 
    \begin{align*}
        u_e(0) = \mu_e + \int_0^1 \frac{w_e(y)}{c_e(y)} \ue^{-D_{\lambda, e}(y)} \, \ud y
    \end{align*}
    for all $e \in E$. Set $p \coloneqq (\ue^{D_{\lambda, e}(1)})_{e \in E}$ and 
    \begin{align}
        q \coloneqq \bigg(\int_0^1 \frac{w_e(y)}{c_e(y)} \ue^{-D_{\lambda, e}(y)} \, \ud y \bigg)_{e \in E}
    \end{align}
    and conclude that $f \in D(A)$ if and only if the identity
    \begin{align*}
        \diag(c(1)) \cdot \diag(p) \cdot \mu = \BNB \cdot \diag(c(0)) \cdot \mu + \BNB \cdot \diag(c(0)) \cdot q + \BB kb
    \end{align*}
	holds. Since we have 
	\begin{align*}
		D_{\lambda, e}(1) &= \int_0^1 d_{\lambda, e}(x) \, \ud x = \int_0^1 \frac{\lambda - c_e'(x)}{c_e(x)} \, \ud x \\
        &\geq \int_0^1 \frac{\lambda}{c_{\max}} \, \ud x - \frac{1}{c_{\min}} \sup_{e \in E} \norm{c_e}_{W^{1, 1}} \geq \frac{\lambda}{c_{\max}} - \frac{1}{c_{\min}} \sup_{e \in E} \norm{c_e}_{W^{1, 1}},
	\end{align*}
	it follows that
    \begin{align*}
        \inf_{e \in E} \diag(c(1)) \cdot \diag(p) > \norm{\BNB \cdot \diag(c(0))} 
    \end{align*}
    for a sufficiently large $\lambda > 0$. As $\diag(c(1)) \cdot \diag(p)$ is a diagonal operator,  Lemma~\ref{lemma:bounded-below} thus implies that the operator
	\begin{align*}
		\diag(c(1)) \cdot \diag(p) - \BNB \cdot \diag(c(0)) \in \calL(\ell^1(E))
	\end{align*}
    is invertible for sufficiently large $\lambda > 0$. Altogether, we conclude that the above equation can be uniquely solved for some $\mu \in \ell^1(E)$ which concludes the proof.
\end{proof}

\begin{remark}
    It shall be noted in terms of Lemma~\ref{lemma:lambda-C-is-surjective} that each $\lambda > 0$ is actually contained in the resolvent set of $C$ and the identity in~\eqref{eq:resolvent-c} states the values of the resolvent. This was not shown yet, since it is not necessary at this point, but will be apparent, since we prove next that $C$ is the generator of a positive, contractive $C_0$-semigroup on $X$.
\end{remark}

\begin{proposition} \label{proposition:well-posedness-of-network-flow-without-perturbation}
	The operator $C$ is the generator of a positive, contractive $C_0$-semigroup on $X$.
\end{proposition}

\begin{proof}
	The operator $C$ is defined on the domain $D(A)$. We begin with proving the density of $D(A)$ in $X$. Observe that 
    \begin{equation*}
        \begin{aligned}
            D(A) \coloneqq \set{f = ((u_e)_{e \in E}, \, (b_v)_{v \in B}) \in X &: \\ & \hspace{-2cm}u_e \in W^{1,1} ((0, 1)) \text{ satisfies } \eqref{equation:boundary-conditions-buffer} \text{ and } \eqref{equation:boundary-conditions-no-buffer}}.
        \end{aligned}
    \end{equation*}
    is dense in $X$ if and only if for any given $\left(b_v\right)_{v \in B} \in \ell^1(B)$ there exists a solution of the linear system of boundary conditions~\eqref{equation:boundary-conditions-buffer}~and~\eqref{equation:boundary-conditions-no-buffer}, where we regard the boundary values $u_e(0)$ and $u_e(1)$ as the unknowns (all other parameters are fixed by the given network). This second statement in turn can be verified, by the following argument: Given any fixed $\left(b_v\right)_{v \in B} \in \ell^1(B)$, the equations~\eqref{equation:boundary-conditions-buffer} uniquely determine $u_e(1)$ for every edge $e \in E$ that leaves a vertex with buffer. We assign arbitrary values to $u_e(0)$ for all edges $e \in E$ and compute the values $u_e(1)$ for the edges that leave a vertex without buffer using~\eqref{equation:boundary-conditions-no-buffer}. Hence, $D(A)$ is dense in $X$.
    
    Moreover, $C$ is dispersive by Lemma~\ref{lemma:C-is-dispersive} and $\lambda - C$ is surjective for some sufficiently large $\lambda > 0$ by Lemma~\ref{lemma:lambda-C-is-surjective}. Thus, Phillips' theorem implies that $C$ generates a positive, contractive $C_0$-semigroup on $X$.
\end{proof}

We are now in position to show that the operator $A$, which describes our transport problem, generates a $C_0$-semigroup as well. For that purpose, we shall leverage one of the perturbation theorems that we presented in Section~\ref{section:relatively-compact-perturbations-of-semigroup-generators}, since $C$ is the generator of a strongly continuous semigroup. To be precise, the following two lemmas allow us to apply Corollary~\ref{corollary:compact-perturbation-stochastic} to the semigroup generated by $C$ and the perturbation $A - C$. 

\begin{lemma}
	\label{lemma:functional-1-in-kernel-of-A}
	One has $\dual{Af, \one} = 0$ for all $f \in D(A)$.
\end{lemma}

\begin{proof}
	Let $f \in D(A)$. Then we obtain
	\begin{align*}
		\dual{Af, \one} &= \sum_{e \in E} \int_0^1 \frac{\ud}{\ud x} (c_e u_e)(x) \, \ud x + \sum_{v \in B} -k_v b_v + \sum_{e \in E} \phi_{ve}^+ c_e(0) u_e(0)\\
		&=\sum_{e\in E} c_e(1) u_e(1) - c_e(0) u_e(0) + \sum_{v \in B} -k_v b_v + \sum_{e \in E} \phi_{ve}^+ c_e(0) u_e(0).
	\end{align*}
	Using the normalization $\sum_{e \in E} w_{ve} = 1$ for each $v \in V$ and the fact that the outgoing incidence matrix $\Phi^-$ is column-stochastic, it follows that
	\begin{align*}
		\dual{Af, \one} &= \sum_{v \in V} \sum_{e \in E} \phi^-_{ve} c_e(1) u_e(1) - \sum_{e \in E} c_e(0) u_e(0) \\ & \quad - \sum_{v \in B} \sum_{e \in E} w_{ve} k_v b_v + \sum_{v \in B} \sum_{e \in E} \phi_{ve}^+ c_e(0) u_e(0).
	\end{align*}
	By substituting the boundary conditions \eqref{equation:boundary-conditions-buffer} and \eqref{equation:boundary-conditions-no-buffer} in the first sum of the above expression, one obtains $\dual{Af, \one} = 0$.
\end{proof}

The next lemma shows that the perturbation of the generator $C$ is indeed compact. 

\begin{lemma} \label{lemma:compact-perturbation}
    The operator $A - C \colon D(A) \to X$ is compact. 
\end{lemma}

\begin{proof}
    Since $B$ is at most countable, one can write $B = \set{b_m : m \in \N}$. We set $B_n \coloneqq \set{b_m : m \geq n}$, $n \in \N$, and consider the finite rank operators 
    \begin{align*}
        K_n f \coloneqq \bigg(0, \bigg(\bigg(\sum_{e \in E} \phi_{ve}^+ c_e(0) u_e(0) \bigg)_{v \in B \setminus B_{n + 1}}, (0)_{v \in B_{n + 1}} \bigg) \bigg), \qquad D(K_n) = D(A).
    \end{align*}
    Then 
    \begin{align*}
        (A - C)f - K_n f = \bigg(0, \bigg((0)_{v \in B \setminus B_{n + 1}}, \bigg(\sum_{e \in E} \phi_{ve}^+ c_e(0) u_e(0) \bigg)_{v \in B_{n + 1}} \bigg) \bigg)
    \end{align*}
    and, thus, 
    \begin{align*}
        \norm{(A - C)f - K_n f} &\leq \sum_{v \in B_{n + 1}} \sum_{e \in E} \phi_{ve}^+ c_e(0) \abs{u_e(0)} \lesssim \sum_{v \in B_{n + 1}} \sum_{e \in E} \phi_{ve}^+ c_e(0) \norm{f}_{D(A)}
    \end{align*}
    for all $f = ((u_e)_{e \in E}, \, (b_v)_{v \in B}) \in D(A)$. In particular, Fubini's theorem yields
    \begin{align} \label{eq:operator-norm-estimate}
        \norm{(A - C) - K_n}_{D(A) \to X} \lesssim \sum_{e \in E} c_e(0) \sum_{v \in B_{n + 1}} \phi_{ve}^+. 
    \end{align}
    Since $G$ is a locally finite graph, for each $e \in E$ one has $\sum_{v \in B_{n + 1}} \phi_{ve}^+ = 0$ for all sufficiently large $n \in \N$. Therefore, it follows from~\eqref{eq:operator-norm-estimate} together with the dominated convergence theorem that $\norm{(A - C) - K_n}_{D(A) \to X} \to 0$ as $n \to \infty$. Since each operator $K_n$ has finite rank, it follows that the operator $A - C \colon D(A) \to X$ is indeed compact. 
\end{proof}

Finally, we are in position to prove our main generation result of this section. 

\begin{theorem} \label{theorem:well-posedness-of-network-flow-problem}
	The operator $A$ generates a stochastic $C_0$-semigroup on $X$.
\end{theorem}

\begin{proof}
    The operator $A - C$ is positive, and Lemma~\ref{lemma:compact-perturbation} implies that it is compact as well. Moreover, Proposition~\ref{proposition:well-posedness-of-network-flow-without-perturbation} shows that $C$ generates a positive $C_0$-semigroup on $X$. Furthermore, one has $\one \in \ker A' = \ker(C + (A - C))'$ by Lemma~\ref{lemma:functional-1-in-kernel-of-A}. Thus, the assertion follows from Corollary~\ref{corollary:compact-perturbation-stochastic}. 
\end{proof}

\section{Long-term behaviour of the flow} \label{section:long-term-behaviour-of-the-flow}

In the previous section, it was shown that the buffered network flow equations are well-posed on the AL-space $X = \ell^1(E; L^1([0, 1])) \times \ell^1(B)$ and that their solutions are given by a stochastic $C_0$-semigroup with generator $A$, which will be denoted by $(T(t))_{t \geq 0}$ throughout this section. In this final section, we discuss the long-term behaviour of these solutions. It is clear that we generally cannot expect convergence without any further assumptions on the underlying graph. In particular, if one wants to prove convergence to equilibrium, one has to guarantee that the mass flow can balance out in the entire network over time.

Recall that a \emph{directed path} $p$ in a directed graph $G = (V, E)$ from the vertex $v \in V$ to the vertex $w \in V$ is a tuple $(e_1, \dots, e_\ell)$ of edges $e_1, \dots, e_\ell \in E$ such that the tail of $e_1$ is $v$, the head of $e_\ell$ is $w$ and the head of $e_i$ is the tail of $e_{i + 1}$ for all $i = 1, \dots, \ell - 1$. In this case, $v$ is called the \emph{starting point} and $w$ the \emph{endpoint} of $p$ and $\ell$ is called the \emph{length} of $p$. The graph $G = (V, E)$ is called \emph{strongly connected} if for all $v, w \in V$ there is a directed path of finite length between them. The following proposition characterises the irreducibility of the transposed adjacency matrix in terms of the underlying graph (see e.g.~\cite[Proposition~4.9]{Dorn2008}). 

\begin{proposition} \label{proposition:adjacency-matrix-irreducible}
    Let $G$ be a locally finite graph. Then $G$ is strongly connected if and only if $\bbB$ is irreducible.
\end{proposition}

However, it turns out that the strong connectedness of the graph is not quite strong enough to obtain convergence of the $C_0$-semigroup, at least in infinite networks. This is the content of the first main result of this section.

\begin{theorem} \label{theorem:strong-convergence-of-the-flow}
	Let $G$ be strongly connected. Then the following assertions are equivalent:
    \begin{enumerate}[\upshape (i)]
        \item The operators $T(t)$ converge strongly as $t \to \infty$.
        \item $1$ is an eigenvalue of the transposed adjacency matrix $\bbB$ of the graph $G$, i.e., $1 \in \sigma_p(\bbB)$. 
    \end{enumerate}
\end{theorem}

\begin{remark}
    Theorem~\ref{theorem:strong-convergence-of-the-flow} is a special case of the main results from \cite{Dorn2009}: More precisely, it is shown in \cite[Theorem~1]{Dorn2009} that if $G$ is strongly connected and $1$ is an eigenvalue of its adjacency matrix $\bbB$, then the semigroup is asymptotically strongly periodic. However, Theorem~\ref{theorem:strong-convergence-of-the-flow} shows that the existence of buffers in the network yields even the strong convergence of the semigroup $(T(t))_{t \geq 0}$ and not just the mere asymptotic periodicity. On the other hand, it is shown in \cite[Lemma~9]{Dorn2009} that the asymptotic periodicity of $(T(t))_{t \geq 0}$ already implies that $1 \in \sigma_p(\bbB)$. 
\end{remark}

As a first step towards the proof of Theorem~\ref{theorem:strong-convergence-of-the-flow}, we prove the following characteristic equation (cf.~\cite[Theorem~3.2]{Dorn2008} for a result in a similar spirit). In a more general context, this result can be seen as a special case of \cite[Proposition~5.10]{Dobrick2023c}. 

\begin{lemma} \label{lemma:characteristic-equation}
    One has $0 \in \sigma_p(A)$ if and only if $1 \in \sigma_p(\bbB)$. Moreover, one has $\dim \ker A \leq \dim \fix \bbB$. 
\end{lemma}

\begin{proof}
    Let $f = ((u_e)_{e \in E}, (b_v)_{v \in B}) \in D(A)$. Since the velocities are uniformly bounded, $f$ is then a solution to the equation $A f = 0$ if and only if there exist $(\alpha_e)_{e \in E} \in \ell^1(E)$ such that
    \begin{align*}
        u_e(x) = \frac{\alpha_e}{c_e(x)} \quad \text{and} \quad b_v = \frac{1}{k_v} \sum_{e \in E} \phi_{ve}^+ \alpha_e
    \end{align*}
    for almost every $x \in [0, 1]$. By Lemma~\ref{lemma:characterization-domain-of-A}, this is equivalent to the sequence $(\alpha_e)_{e \in E}$ solving the eigenvalue problem
    \begin{align*}
        (\alpha_e)_{e \in E} = \bbB (\alpha_e)_{e \in E}
    \end{align*}
    on the space $\ell^1(E)$. Therefore, $0 \in \sigma_p(A)$ if and only if $1 \in \sigma_p(\bbB)$. The remaining claim is evident. 
\end{proof}

Next, we compute the fixed space of the flow semigroup $(T(t))_{t \geq 0}$, explicitly. 

\begin{proposition} \label{proposition:fixed-space-of-flow}
    Let $G$ be strongly connected and suppose that $1 \in \sigma_p(\bbB)$. Then $\ker A$ is one-dimensional and spanned by a vector $0 \ll f \in \fix (T(t))_{t \geq 0}$.
\end{proposition}

\begin{proof}
    Since $G$ is strongly connected, Proposition~\ref{proposition:adjacency-matrix-irreducible} implies that the transposed adjacency matrix $\bbB$ is irreducible. Moreover, $\bbB$ is a stochastic operator on $\ell^1(E)$ with $1 \in \sigma_p(\bbB)$ and, thus, \cite[Proposition~2.1]{Keicher2008} yields that the eigenspace of $\bbB$ to the eigenvalue $1$ is one-dimensional and spanned by a strictly positive eigenvector $0 \ll (w_e)_{e \in E} \in \ell^1(E)$ (cf.\ also  \cite[Theorem~V.5.2(i)]{Schaefer1974}).
    
    If $f = ((u_e)_{e \in E}, (b_v)_{v \in B}) \in D(A)$, then it follows from Lemma~\ref{lemma:characterization-domain-of-A} that $f \in \ker A$ if and only if there exists $\lambda \in \C$ such that
    \begin{align*}
        u_e = \lambda \frac{w_e}{c_e} \quad \text{and} \quad b_v = \frac{\lambda}{k_v}\sum_{e \in E} \phi^+_{ve} w_e
    \end{align*}
    for all $e \in E$ and $v \in B$. In particular, $\ker A$ is one-dimensional and spanned by a strictly positive vector, e.g., by $f \coloneqq ((\frac{w_e}{c_e})_{e \in E}, (\frac{1}{k_v}\sum_{e \in E} \phi^+_{ve} w_e)_{v \in B}) \in \ker A$.
\end{proof}

We have seen in the above computation that, in the case of an equilibrium, the flow actually behaves as if there would not be any buffers at all (cf.\ Remark \ref{remark:characterization-domain-of-A-without-buffers}). Or, putting it the other way around, every buffer contains exactly the right amount of mass such that its emitted mass (proportional to the fill level) equals the incoming mass. Next, we prove that the buffered network flow semigroup is irreducible. 

\begin{proposition} \label{proposition:irreducibility-of-flow}
	Let $G$ be strongly connected and suppose that $1 \in \sigma_p(\bbB)$. Then $(T(t))_{t \geq 0}$ is irreducible on $X$.
\end{proposition}

For the proof of Proposition~\ref{proposition:irreducibility-of-flow}, we shall facilitate the following more general result. Let $(S(t))_{t \geq 0}$ be a positive semigroup on a Banach lattice $X$. Recall that a vector $x \in X_+$ is called a \emph{super fixed point} of the semigroup $(S(t))_{t \geq 0}$ if $S(t) x \geq x$ for all $t \geq 0$. For the notion of Banach lattices with order continuous norm and projection bands, we refer to \cite[Chapter~2]{Wnuk1988} and \cite[Definition~II.2.8]{Schaefer1971}, respectively. 

\begin{proposition} \label{proposition:fixed-space-irreducible-semigroup}
    Let $(S(t))_{t \geq 0}$ be a positive, bounded $C_0$-semigroup with generator $B$ on a Banach lattice $X$ with order continuous norm and suppose that every super fixed point of $(S(t))_{t \geq 0}$ is a fixed point. If $\ker B$ is one-dimensional and spanned some vector $0 \ll x \in X$, then $(S(t))_{t \geq 0}$ is irreducible on $X$.
\end{proposition}

\begin{proof} 
    Aiming for contradiction, we assume there exists a closed non-trivial ideal $I \subseteq X$ that is invariant under the action of $(S(t))_{t \geq 0}$. Since $X$ has order continuous norm, $I$ is a projection band by \cite[Corollary~2.4.4]{MeyerNieberg1991} (or, alternatively,  \cite[Theorem~17.17]{Zaanen1997}). In particular, one has the decomposition $X = I \oplus I^\perp$. Denote the band projections onto $I$ and $I^\perp$ by $P$ and $Q$, respectively. Since $I$ is invariant under the action of $(S(t))_{t \geq 0}$, one has $S(t) P x \in I$ and, thus,
    \begin{align*}
        S(t) P x = P S(t) P x \leq P S(t) x = P x \qquad \text{for all } t \geq 0. 
    \end{align*}
    As $x = Px + Qx$, this implies $Q x \leq S(t) Q x$ for all $t \geq 0$, i.e., $Qx$ is a super fixed point of $(S(t))_{t \geq 0}$ and, consequently, a fixed point by hypothesis. Hence, $Px$ is a fixed point, too. As $x$ is a quasi-interior point of $X_+$, the principal ideal $X_x = \bigcup_{n \in \N} [-nx, nx]$ is dense in $X$. Thus, $Px = 0$ would imply that $P = 0$ and, analogously, $Qx = 0$ would imply $Q = 0$, which is absurd, since both $I$ and $I^\perp$ are non-trivial. Hence, we infer that $P x \neq 0$ and $Q x \neq 0$. Moreover, $Px$ and $Qx$ are disjoint, i.e., $Qx \wedge Px = 0$, and, thus, linearly independent. Therefore,
    \begin{align*}
        \dim \fix (S(t))_{t \geq 0} \geq \dim \nspan \set{Px, Qx} = 2.
    \end{align*}
    However, $\ker B = \fix (S(t))_{t \geq 0}$ and $\dim \ker B = 1$, which is a contradiction. 
\end{proof}

We are now in position to prove Proposition~\ref{proposition:irreducibility-of-flow}.

\begin{proof}[Proof~of~Proposition~\ref{proposition:irreducibility-of-flow}]
	By Proposition~\ref{proposition:fixed-space-of-flow}, $\ker A$ is one-dimensional and spanned by a strictly positive fixed vector. Moreover, $A$ is the generator of the stochastic semigroup $(T(t))_{t \geq 0}$ by Theorem~\ref{theorem:well-posedness-of-network-flow-problem} on the AL-space $X$. As each AL-space has order continuous norm, by Proposition~\ref{proposition:fixed-space-irreducible-semigroup}, it suffices to prove that every super fixed point of $(T(t))_{t \geq 0}$ is a fixed point. 

    So let $0 \leq f \in X$ be a super fixed point of $(T(t))_{t \geq 0}$. Suppose there exists some $t_0 \geq 0$ such that $T(t_0) f > f$. Since $X$ has a strictly monotone norm (see \cite[Remark~7.7]{Eisner2015}), i.e., for all $g, h \in X_+$ with $g \leq h$ and $g \neq h$, one has $\norm g < \norm h$, this would imply that $\norm{T(t_0) f} > \norm{f}$. However, this would contradict the contractivity of $(T(t))_{t \geq 0}$. Thus, $T(t) f = f$ for all $t \geq 0$, i.e., $f \in \fix (T(t))_{t \geq 0}$. Hence, $(T(t))_{t \geq 0}$ is irreducible on $X$.
\end{proof}

Next, we recall an observation from the proof of \cite[Theorem~6.1]{Gerlach2019}. For the convenience of the reader, we include its simple proof. Recall that if $(\Omega, \Sigma, \mu)$ is a $\sigma$-finite measure space, then $\omega \in \Omega$ is called an \emph{atom} if $\set \omega \in \Sigma$ and $\mu(\set \omega) > 0$. 

\begin{lemma} \label{lemma:atom-implies-partially-integral}
    Let $(\Omega, \Sigma, \mu)$ be a $\sigma$-finite measure space and let $(T(t))_{t \geq 0}$ be a bounded, positive semigroup on $L^p(\Omega)$, $1 \leq p < \infty$. Further, suppose that $(T(t))_{t \geq 0}$ is irreducible and that $\Omega$ has an atom $\omega \in \Omega$. Then there exists $t_0 > 0$ and a non-zero compact operator $K \in \calL(L^p(\Omega))$ such that $0 \leq K \leq T(t_0)$.
\end{lemma}

\begin{proof}
    If $\dim L^p(\Omega) = 1$, then there is nothing to show since $(T(t))_{t \geq 0}$ cannot be trivial. So suppose that $\dim L^p(\Omega) \geq 2$. Let $a \colon \Omega \to \R$ be the indicator function of the singleton $\set{\omega}$. Since $(\Omega, \Sigma, \mu)$ is $\sigma$-finite, one has $\mu(\set \omega) < \infty$ and thus $a \in L^p(\Omega)$. Consider the projection band $B \coloneqq \set{a}^{\perp \perp} \subseteq L^p(\Omega)$ generated by $a$ and let $P \colon L^p(\Omega) \to B$ be the associated band projection onto $B$. Then the orthogonal band $B^\perp$ is not trivial since $B$ is one-dimensional and $\dim E \geq 2$. So, by irreducibility, $B^\perp$ cannot be invariant under the action of $(T(t))_{t \geq 0}$. In particular, there exists $t_0 > 0$ such that $P T(t_0) \neq 0$. As $K \coloneqq P T(t_0) \geq 0$ has one-dimensional range, it is compact and one clearly has $K \leq T(t_0)$. 
\end{proof}

We are now in position to prove our first main convergence result of this section. 

\begin{proof}[Proof~of~Theorem~\ref{theorem:strong-convergence-of-the-flow}]
    (i) $\Rightarrow$ (ii): Since $(T(t))_{t \geq 0}$ is a stochastic semigroup on $X$,  the operators $T(t)$ converge to a non-zero projection as $t \to \infty$. Thus, $1 \in \sigma_p(T(t))$ for all $t \geq 0$, which implies $0 \in \sigma_p(A)$. However, this is equivalent to $1 \in \sigma_p(\bbB)$ by Lemma~\ref{lemma:characteristic-equation}. 

    (ii) $\Rightarrow$ (i): 
    By Theorem~\ref{theorem:well-posedness-of-network-flow-problem}, $(T(t))_{t \geq 0}$ is a stochastic semigroup on $X$. Moreover, $(T(t))_{t \geq 0}$ is irreducible by Proposition~\ref{proposition:irreducibility-of-flow}. Furthermore, the state space $X = \ell^1(E; L^1([0, 1])) \times \ell^1(B)$ is isometrically isomorphic to an $L^1$-space on a $\sigma$-finite measure space and clearly contains an atom, since we assumed that there exists at least one buffer. So, Lemma~\ref{lemma:atom-implies-partially-integral} implies that there exists $t_0 > 0$ and a non-zero compact operator $K$ on $X$ such that $0 \leq K \leq T(t_0)$.
    Thus, Proposition~\ref{proposition:fixed-space-of-flow} and \cite[Corollary~4.4]{Gerlach2019} yield the claim. 
\end{proof}

Under the additional assumption that the network is finite, we obtain the second main result of this section. 

\begin{theorem} \label{theorem:norm-convergence-of-the-flow}
	Let $G$ be strongly connected and finite. Then the operators $T(t)$ converge with respect to the operator norm as $t \to \infty$.
\end{theorem}

For the proof of the above theorem, we shall employ the following lemma, which is an easy consequence of a classical Sobolev embedding. 

\begin{lemma} \label{lemma:generator-compact-resolvent}
    Let $G$ be finite. Then the generator $A$ of the buffered flow semigroup $(T(t))_{t \geq 0}$ has compact resolvent.
\end{lemma} 

\begin{proof}
    Since $A$ is a generator, we have $\rho(A) \neq \emptyset$. By \cite[Proposition~II.4.25]{Engel2000}, the claim is equivalent to the compactness of the canonical embedding $\iota \colon D(A) \to X$, where $D(A)$ is equipped with the graph norm. However, the compactness of $\iota$ follows readily from the Rellich--Kondrachov theorem (see,~e.g.,~\cite[Theorem~6.2]{Adams1975}) and the fact that the graph $G$ is finite. 
\end{proof}

We can now prove that the semigroup $(T(t))_{t \geq 0}$ is convergent even with respect to the operator norm, given that the network is finite. 

\begin{proof}[Proof~of~Theorem~\ref{theorem:norm-convergence-of-the-flow}]
    By Theorem~\ref{theorem:well-posedness-of-network-flow-problem}, $(T(t))_{t \geq 0}$ is a stochastic semigroup on the AL-space $X$. Moreover, $(T(t))_{t \geq 0}$ is irreducible by Proposition~\ref{proposition:irreducibility-of-flow} and, as in the proof of Theorem~\ref{theorem:strong-convergence-of-the-flow}, it follows that there exists $t_0 > 0$ and a non-zero compact operator $K$ on $X$ such that $0 \leq K \leq T(t_0)$. Furthermore, the generator $A$ of $(T(t))_{t \geq 0}$ has compact resolvent by Lemma~\ref{lemma:generator-compact-resolvent}. Thus, $0$ is a pole of the resolvent of $A$ by \cite[Corollary~IV.1.19]{Engel2000}. Therefore, \cite[Theorem~1.1]{Glueck2022} yields the claim.
\end{proof}

\begin{remark}
    In view of \cite[Theorem~4.10]{Dorn2008}, one would expect that one can drop the assumption of the finiteness of the network in Theorem~\ref{theorem:norm-convergence-of-the-flow}, if one assumes that the transposed adjacency matrix $\bbB$ is quasi-compact on $\ell^1(E)$. More precisely, we would expect that the operators $T(t)$ convergence to a non-zero projection with respect to the operator norm as $t \to \infty$ if and only if $G$ has an attractor (see \cite[Proposition~4.8]{Dorn2008}), given that $G$ is strongly connected. However, it is unclear to the authors how to prove this statement (cf.~\cite[Open Question~5.2]{Dobrick2023c}). 
\end{remark}

%
%
\bibliographystyle{plain}
\bibliography{literature}

\end{document}